\documentclass[12pt, reqno]{amsart}
\usepackage{amsmath, amsthm, amscd, amsfonts, amssymb, graphicx, color}
\usepackage[bookmarksnumbered, colorlinks, plainpages]{hyperref}
\hypersetup{colorlinks=true,linkcolor=red, anchorcolor=green, citecolor=cyan, urlcolor=red, filecolor=magenta, pdftoolbar=true}

\newtheorem{theorem}{Theorem}[section]
\newtheorem{lemma}[theorem]{Lemma}
\newtheorem{proposition}[theorem]{Proposition}
\newtheorem{corollary}[theorem]{Corollary}
\theoremstyle{definition}
\newtheorem{definition}[theorem]{Definition}
\newtheorem{example}[theorem]{Example}

\theoremstyle{remark}
\newtheorem{remark}[theorem]{Remark}
\numberwithin{equation}{section}

\begin{document}

\setcounter{page}{1}

\title[On biamenability of Banach algebras]{On biamenability of Banach algebras}
\author[S. Barootkoob]{S. Barootkoob}
\address{Faculty of Basic Sciences, University of Bojnord, P.O. Box 1339, Bojnord, Iran}
\email{s.barutkub@ub.ac.ir}




\begin{abstract}
In this paper, we introduce the concept of biamenability of Banach algebras and we show that despite the apparent similarities between   amenability and biamenability of Banach algebras, they lead to very different, and
somewhat opposed, theories. In this regard,  we show that commutative Banach algebras such as $\mathbb{R}$ and $  \mathbb{C}$   tend
to lack biamenability, while they may be amenable and highly noncommutative Banach algebras such as $B(H)$ for an infinite dimensional Hilbert space $H$ tend
to be biamenable, while they are not amenable. Also, we show that although the unconditional unitization of an amenable  Banach algebra is amenable but in general unconditional unitization of a Banach algebra is not biamenable.

This concept is used for finding the character space of some Banach algebras. For example we will show that for each infinite dimensional Hilbert space $H$, and each integer $n\geq 0$,  $B(H)^{(2n)}$  and the module extension Banach algebra $B(H)\oplus B(H)^{(2n)}$ have empty spectrum.

Finally, we introduce the concept of biamenability for a pair of Banach algebras and we study the relation between biamenability of a pair of Banach algebras and amenability of them.
\end{abstract}
 \maketitle

\section{Introduction}
A derivation from a Banach algebra $A$ to a
  Banach $A$-bimodule $X$ with the continuous module operations is a bounded linear mapping $d:A\rightarrow X$
  such that  $$d(ab)=d(a)b+ad(b)\ \ \ (a,b\in A).$$ For each $x\in
  X$ the  mapping $\delta_x:a\rightarrow ax-xa$,  $(a \in A)$  is a
  bounded derivation, called an inner derivation.\\
Let $X$ be a Banach $A$-bimodule. Then $X^*$ is a dual   Banach $A$-bimodule, by defining $a.f$ and $f.a$, for each $a\in A$ and $f\in X^*$ by$$a.f(x)=f(xa) \ \ \  ,\ \ \ \ f.a(x)=f(ax)\hspace{1cm}(x\in X).$$
Similarly, the higher duals $X^{(n)}$ can be made into Banach $A$-bimodules in a natural fashion.   

  A Banach algebra $A$ is called amenable if for each Banach $A$-bimodule $X$, the only derivations from $A$ to $X^*$ are inner derivations. For more details about this notion see \cite{Ru}. 
  
Let $A$ be a Banach algebra and $X$ be a Banach $A$-bimodule, a bounded bilinear mapping $D:A\times A \rightarrow X$ is called a biderivation if $D$ is a derivation  with respect to both arguments. That is the mappings $_aD:A\rightarrow X$ and $D_b:A\rightarrow X$  are derivations. Where $$_aD(b)=D(a,b)=D_b(a) \hspace{2cm}(a, b\in A).$$ We denote the space of such biderivations by $BZ^1(A, X)$.

Consider the subspace  $Z(A,X) =\{x\in X: ax=xa \ \  \forall a\in A\}$ of $X$. Then for each $x\in Z(A,X)$, 
the mapping $\Delta_x:A \times A\rightarrow X$ defined by
  $$\Delta_x(a,b)=x[a,b]=x(ab-ba) \ \ \ \ \ \ (a, b\in A)$$
 is a basic example of a biderivation and called an inner biderivation. We denote the space of such inner biderivations by $BN^1(A,X)$. 
 For more applications of biderivations,  see the survey article {\cite[Section 3]{Bre}}. Some algebraic aspects of biderivations on certain algebras were investigated by many authors; see for example \cite{Ben, DW}, where the structures of biderivations on triangular algebras and generalized matrix algebras are discussed, and particularly   the question of whether  biderivations on these algebras are inner, was considered. 

We define the first bicohomology group $BH^1(A,X)$ as follows, $$BH^1(A,X)=\frac{BZ^1(A,X)}{BN^1(A,X)}.$$
 Obviously $BH^1(A,X)=0$ if and only if every biderivation from $A\times A$ to $X$ is an inner biderivation.  Now we are motivated to define the concept of biamenability of Banach algebras as follows.\\
A Banach algebra $A$ is biamenable if for each Banach $A$-bimodule $X$ we have $BH^1(A, X^*) = 0.$

Although one might expect that biderivations and biamenability must run parallel to derivations and amenable Banach algebras what is true is that although there are some external similarities between  them but they lead to very different, and
somewhat opposed, theories. Indeed  we show that commutative Banach algebras tend
to lack biamenability, while highly noncommutative Banach algebras tend
to be biamenable. Thus for instance, the ground algebras $\mathbb{C}$ and $\mathbb{R}$ are not
biamenable (while they are trivially amenable) and $B(H)$, the algebra of all bounded operators
on an infinite dimensional Hilbert space $H$, is biamenable, but not amenable. Moreover, if $H$ is
finite dimensional, it turns out that $B(H)$ is amenable, but it fails to be biamenable.
\section{biamenable  Banach algebras}
 For an example of  a biamenable Banach algebra we commence with the next lemmas.
The following lemma  is similar to Corollary 2.4 of \cite{b}, where it is introduced for a biderivation $D:A\times A\rightarrow A$. 
\begin{lemma}\label{lem1}
For each Banach $A$-bimodule $X$ and each biderivation $D:A\times A\rightarrow X$, $$D(a,b)[c,d]=[a,b]D(c,d)\hspace{1cm}(a, b, c, d\in A).$$ 
\end{lemma} 
 In Proposition 2.1.3 of \cite{Ru}  it is shown that if $A$ has  a bounded approximate identity,
and one of module actions is trivial, then the only derivations from $A$ to $X^*$ are inner derivations. The following lemma introduces a condition   that not only implies the innerness of biderivations but it also forces them to be zero.
 \begin{lemma}\label{lem2}
If a Banach algebra $A$ has  a bounded left approximate identity  and $span\{ab-ba: a,b\in A\}$ is dense in $A$, then for every Banach $A$-bimodule $X$ such that $XA=0$ we have $BZ^1(A,X)=0$.  
 \end{lemma}
 \begin{proof}
 Let $(e_\alpha)$ be a left approximate identity of $A$. Lemma \ref{lem1} says that for each $D\in BZ^1(A,X)$,$$[a,b]D(c,d)=0\hspace{1cm}(a, b, c, d\in A).$$ 
So by density we have $aD(b,c)=0$, for each $a, b, c\in A$.
 Now since $XA=0$ we have  $$\begin{array}{rcl}
 D(a,b)&=&\lim_\alpha D(e_\alpha a,b)\\
 &=&\lim_\alpha[e_\alpha D(a,b)+D(e_\alpha,b)a]\\
 &=&0.
 \end{array}$$
That is $BZ^1(A,X)=0$.
 \end{proof}
A very similar proof may be applied if  $A$ has a  right approximate identity and the left module action is trivial.\\
The following lemma introduces a condition  that under which some biderivations are inner. This result will lead to a condition implying biamenability of a Banach algebra. We shall see  that $B(H)$, for each  infinite dimensional Hilbert space $H$, satisfies this condition.
 \begin{lemma}\label{lem3}
 If $A$ is unital and $A=span\{ab-ba: a,b\in A\}$, then for every  unital Banach $A$-bimodule $X$, every  biderivation $D:A\times A\rightarrow X$ is an inner  biderivation.
 \end{lemma}
 \begin{proof}
 Let $D$ be a biderivation and $e$ be the identity of $A$. Since $A=span\{ab-ba: a,b\in A\}$,  there exist  $a_i$ and $b_i$ in $A$ such that $e=\sum_{i}[a_i,b_i]$. So by Lemma \ref{lem1}, for every $a, b\in A$, we have
 $$\begin{array}{rcl}
 D(a,b)&=& D(a,b)e\\
&=& D(a,b)\sum_{i}[a_i,b_i]\\
 &=& \sum_{i}D(a,b)[a_i,b_i]\\
 &=&\sum_i[a,b]D(a_i,b_i)\\
 &=&[a,b]\lambda.
 \end{array}$$
Where $\lambda= \sum_{i}D(a_i,b_i)$. Similarly we have $D(a,b)=\lambda [a,b],$  and so $$\lambda [a,b]=[a,b]\lambda\hspace{1cm}(a\in A, b\in B).$$
Now  since $A=span\{ab-ba: a,b\in A\}$, $\lambda\in Z(A,X)$.  So $D=\Delta_\lambda$ is an inner  biderivation. 
 \end{proof}
 \begin{lemma}\label{lem4}
 If a Banach algebra $A$ has  a bounded  approximate identity  and $span\{ab-ba: a,b\in A\}$ is dense in $A$,
  then the following statements are equivalent.
\begin{itemize}
 \item[(i)] $A$ is biamenable.
 \item[(ii)] $BH^1(A,X^*)=0$, for every left approximately unital Banach $A$-bimodule $X$.
 \item[(iii)] $BH^1(A,X^*)=0$, for every right approximately unital Banach $A$-bimodule $X$.
 \item[(iv)] $BH^1(A,X^*)=0$, for every  approximately unital Banach $A$-bimodule $X$.
 \end{itemize}
 \end{lemma}
 \begin{proof}
We only prove (i) is equivalent to (ii). The equivalence of (i) and (iii) is similar and then the equivalence of (i) and (iv) is obvious.\\
Clearly if $A$ is biamenable then (ii) is true. Now let  $BH^1(A,Y^*)=0$, for every left approximately unital Banach $A$-bimodule $Y$ and let $X$ be a Banach $A$-bimodule. Then Corollary 2.9.26 of \cite{D} implies that $X_0=AX$ is a left approximately unital closed submodule of $X$. Also $A(\frac{X}{X_0})=0$ and so $(\frac{X}{X_0})^*A=0$. Therefore Lemma \ref{lem2} says that $BZ^1(A,X_0^\perp)= BZ^1(A,(\frac{X}{X_0})^*)=0$.\\
Let $D\in BZ^1(A,X^*)$ and $J:X_0\rightarrow X$ be the inclusion mapping. Then $J^*\circ D\in BZ^1(A,X_0^*)$ and by assumption $J^*\circ D=\Delta_{\phi_0}$, for some $\phi_0\in Z(A,X_0^*)$. Now the equation $X^*=X_0^*\oplus X_0^\perp$, which is implied from Theorem 4.9 of \cite{R}, shows that there exists an extension $\phi$ of $\phi_0$ such that $\phi\in Z(A,X^*)$. \\
Define $D_0=D-\Delta_\phi$. Then $D_0\in BZ^1(A, X_0^\perp)=0$ and so $D=\Delta_\phi$.
 \end{proof}
 A similar result of the previous  lemma in the area of amenability is given in Proposition 2.1.5 of \cite{Ru}.\\
 Now  combination of the Lemmas \ref{lem3} and \ref{lem4} gives the following theorem that leads to a condition for biamenability of Banach algebras and then we can find some examples of biamenable Banach algebras which are not amenable. 
 \begin{theorem}\label{cor}
Each unital Banach algebra  $A$ with $A=span\{ab-ba: a,b\in A\}$, is biamenable.
 \end{theorem}
\begin{corollary}
If  $A=span\{ab-ba: a, b\in A\}$  and $A$ has an identity, then the only biderivation $D:A\times A\rightarrow A^*$ is zero.
\end{corollary}
\begin{proof}
Let $D:A\times A\rightarrow A^*$ be a biderivation. By Theorem \ref{cor} $A$ is  biamenable, so $D$ is inner. That is there exists an $f\in Z(A,A^*)$ such that for every $a, b\in A$, $D(a, b)=f[a, b]$.
Now $$\langle f,[a,b]\rangle=\langle fa-af,b\rangle=0.$$
Hence our assumption implies that $f=0$ and so $D=0$.
\end{proof}
Every bounded bilinear mapping $f: X\times Y\rightarrow Z$  on normed spaces $X, Y$ and $Z$,  has two natural extensions $f^{***}$ and  $f^{t***t}$ from $X^{**}\times Y^{**}$  to $Z^{**}$ as follows. \\
We define the adjoint  $ f^*:{Z}^*\times {X} \rightarrow {Y}^*$ of $f$  by \[\langle f^*(z^*, x),y\rangle=\langle z^*,f(x,y)\rangle,\] where $x\in {X}, y\in {Y}\ \ {\rm and}\ \  z^*\in {Z}^*.$
We then define 
  $f^{**}=(f^*)^*$ and $f^{***}=(f^{**})^*$. 
 Let  $f^{t}: {Y} \times {X}
\longrightarrow {Z}$ be the flip map of $f$  which is defined by $f^{t}(y,x)= f(x,y)\  ( x \in {X}, y
\in {Y}).$   If we continue the latter process with $f^t$ instead of $f$, we come to the bounded bilinear mapping $f^{t***t}: {X}^{**}\times {Y}^{**}\rightarrow Z^{**}$.\\
Where $\pi$ is the multiplication of a Banach algebra $A$, $\pi^{***}$ and $\pi^{t***t}$ are actually the first and second Arens products,  which are denoted by
 $\square$ and $\lozenge$, respectively.  For more detailes see \cite{A} and \cite{Ar}.
 
 Now we give some examples of biamenable Banach algebras.
\begin{example}\label{ex}
According to Lemma 5.8 of \cite{Z}, since for each infinite dimensional Hilbert space $H$ and every integer $n\geq 0$,
$$B(H)^{(2n)}=span\{au-ua: a\in B(H), u\in B(H)^{(2n)}\}.$$
So Theorem \ref{cor} help us to find some biamenable Banach algebras such as the Banach algebra
$B(H)^{(2n)}$ and the module extension Banach algebra $B(H)\oplus B(H)^{(2n)}$ with the  product and norm as follows. $$(a,u)(b,v)=(ab,av+ub),\ \ \ \ \|(a,u)\|=\|a\|+\|u\|, \ \ \ \ \ (a\in B(H), u\in B(H)^{(2n)}).$$  Although $B(H)$ is not amenable in general.
Note that since  $\{au-ua: a\in B(H), u\in B(H)^{(2n)}\}$ is a subset of  $\{uv-vu: u, v\in B(H)^{(2n)}\}$ and $\{[a,v]-[b,u]: a,b\in B(H), u,v\in B(H)^{(2n)}\}$.
Therefore the commutators span the whole of $B(H)^{(2n)}$ and $B(H)\oplus B(H)^{(2n)}$. \\
 Also similar to  last corollary we can show that  the only biderivation from $B(H)\times B(H)$ to $B(H)^{(2n+1)}$ is zero. 
 
A similar method as  Lemma 5.7 of \cite{Z}, can show that for the Banach algebra $K(H)$ of compact operators on $H$, $$K(H)=span\{ku-uk: k\in K(H): u\in B(H)\}.$$ So similarly $B(H)\oplus K(H)$ is biamenable. Although Remark 5.10 of \cite{Z} says that it is not amenable.
\end{example}
Let $G$ be a locally compact group. $m\in L^\infty(G)^*$ is a mean on $L^\infty(G)$ if $m(1)=\|m\|=1$. A mean $m$ on $L^\infty(G)$ is called a left invariant mean if for each $x\in G$ and $g\in L^\infty(G)$, $m(\delta_x*g)=m(g)$.
$G$ is called amenable if there is a left invariant mean on $L^\infty(G)$.\\
Consider $L^\infty(G)$ as an $L^1(G)$-bimodule with the left and right module actions
$$\pi_\ell:L^1(G)\times L^\infty(G) \rightarrow L^\infty(G)\quad \pi_r:L^\infty(G)\times L^1(G)\rightarrow L^\infty(G),$$ defined by $\pi_\ell(f,g)=f*g$ and  $\pi_r(g,f)=(\int_G f) g$. Then we have the following proposition.
\begin{proposition}
Let $G$ be a locally compact group such that $Z(L^1(G),L^\infty(G)^*)$ containes an element $n$ such that $n(1)\neq 0$. Then $G$ is amenable.
\end{proposition}
\begin{proof}
Define $|n|$ by $|n|(\phi)=|n(g)|$, for each $g\in L^\infty(G)$. Then $m=\frac{|n|}{|n(1)|}$ is a positive element of $Z(L^1(G),L^\infty(G)^*)$ such that $m(1)=1$. Therefore by {\cite[Proposition 1.1.2]{Ru}}, $m$ is a mean on $L^\infty(G)$.\\
Now we have 
\begin{align*}
\langle m,f*g\rangle&=\langle m,\pi_\ell(f,g)\rangle\\
&=\langle \pi_\ell^*(m,f),g\rangle\\
&=\langle\pi_r^{t*t}(f,m),g\rangle\\
&=\langle m,\pi_r(g,f)\rangle\\
&=\langle m,g\rangle\int_Gf\\
&=\langle m,g\rangle\quad (f\in P(G), g\in L^\infty(G)).\end{align*}
Where $P(G)=\{f\in L^1(G); \|f\|_1=1, f\geq 0\}$. So  {\cite[Lemma 1.1.7]{Ru}} implies that $G$ is amenable.
\end{proof}
A big class of Banach algebras are not biamenable, although they may be amenable.
For example if there exists a non-zero  derivation  $d:A\rightarrow A^{**}$  on a commutative Banach algebra  $A$ such that for some $a, b\in A$, $d(a) \square d(b)\neq 0$, then  $A$ is not biamenable.
Since the map 
$$\begin{array}{rcl}
D:A \times A&\longrightarrow& A^{**}\\
(a,b)&\mapsto& d(a)\square d(b)
\end{array}$$ 
where $\square $ denotes the first Arens product of $A^{**}$, defines a biderivation which is not inner. Also
if there is a Banach $A$-bimodule $X$ such that $Z(A,X)=\{0\}$ and there is a non zero  biderivation from $A\times A$ into  $X^*$, then $A$ is not  biamenable.
Since in this case $Z(A,X^*)=\{0\}$ and so the  only inner biderivation $D:A\times A\rightarrow X^*$ is zero and therefore we have the following proposition.
\begin{proposition}\label{pr3}
Let $\sigma(A)$ be the spectrum of $A$. If $\sigma(A)\neq\emptyset$, then $A$ is not biamenable.
\end{proposition}
\begin{proof}
Let $A$ be a Banach algebra such that $\sigma(A)\neq \emptyset$, $f$ be an element in $\sigma(A)$ and $X$ be a non zero Banach $A$-bimodule with module actions $$ax=0,  \ \ \ xa=f(a) x\hspace{1cm}(a\in A, x\in X).$$ Then $Z(A,X)=\{0\}$. But for a non-zero element $h\in X^*$ 
$$\begin{array}{rcl}
D:A\times A&\rightarrow& X^*\\
(a,b)&\mapsto&f(a) (bh-hb)
\end{array} $$
is a non-zero biderivation. 
\end{proof}
Now by applying  Proposition \ref{pr3} and Theorem 1.3.3 of \cite{M}, we conclude that every unital commutative Banach algebra is not biamenable. For example $\mathbb{C}$, $\mathbb{R}$, $C(\Omega)$,  for each Hausdorff space $\Omega$ and  the group algebra $M(G)$, for each locally compact abelian group $G$ are not biamenable. In the next section we extend this result 
to arbitrary commutative Banach algebras. Also a combination of Example \ref{ex} and Proposition
\ref{pr3} implies the following.
\begin{corollary}
For each integer $n\geq 0$ and each infinite dimensional Hilbert space $H$, $\sigma(B(H)^{(2n)})=\emptyset$ and  $\sigma(B(H)\oplus B(H)^{(2n)})=\emptyset$.
\end{corollary}

\begin{example}
Let $A$ be a Banach space and $\theta$ be a non zero element of $A^*$. Then $A$ is a Banach algebra with the multiplication
$$ab=\theta(a) b,\hspace{1cm}(a, b\in A).$$
Now since $\theta\in \sigma(A)$, $A$ with this multiplication is not biamenable.
\end{example}
\section{Some properties}
In this section we study some properties of biamenable Banach algebras and we tend to some another examples of non biamenable Banach algebras.
\begin{theorem}
 Let $A$ be a Banach algebra and consider $\mathbb{C}$ as a Banach $A$-bimodule. If there is a nonzero derivation $d : A \rightarrow\mathbb{C}$, then  biamenability of $A$ implies  amenability of $A$. 
 \end{theorem}
\begin{proof}
 Let $X$ be a Banach $A$-bimodule and $d' : A \rightarrow X^*$ be a bounded derivation. Then 
 $$\begin{array}{rcl}
 D : A\times A &\rightarrow &  X^*\\ (a,b)&\mapsto & d(a)d'(b)
 \end{array}$$
is a bounded biderivation and so there is $f\in Z(A,X^*)$ such that $$d(a)d'(b) = D(a,b) = f[a,b]\quad (a,b\in A).$$ Therefore for every $b\in A$ and for some $a \in A$ such that $d(a)\neq 0$ we have $d'(b) = \delta_{-\frac{fa}{d(a)}}(b)$. 
\end{proof}
\begin{example}
  Let $\mathbb{D} =\{z \in\mathbb{C}; |z|\leq 1\}$ be the unit disc, and $A(\mathbb{D})$ be the disc algebra. We can consider $\mathbb{C}$ as an $A(\mathbb{D})$-bimodule with module actions $\alpha f =\alpha f(0) = f\alpha$ and 
  $$\begin{array}{rcl}
 d : A(\mathbb{D}) &\rightarrow &  \mathbb{C}\\ f &\mapsto & f'(0)
 \end{array}$$ is a nonzero derivation. Therefore since $A(\mathbb{D})$ is not amenable so it is not biamenable.
 \end{example}
We know  that every amenable Banach algebra has an approximate identity  (See Proposition 2.2.1 of \cite{Ru}). A similar result is given in the following.
\begin{proposition}\label{pr1}
If  $A=span\{ab-ba: a,b\in A\}$ and  $A$ is biamenable, then $A$ has a bounded approximate identity.
\end{proposition}
\begin{proof}
 If $A$ is biamenable, then for the biderivation 
 $$\begin{array}{rcl}
 D:A\times A&\rightarrow &A^{**}\\
 (a,b)&\mapsto &[a,b]
 \end{array}$$
 there is $E\in Z(A,A^{**})$ such that for each $a, b\in A$, $E[a,b]=[a,b]$. Now let $(e_\alpha)$ be a bounded net in $A$ which is $w^*$-convergent to $E$. Then we have
 $$\begin{array}{rcl}
\lim_\alpha e_\alpha [a,b]&=&w-\lim_\alpha e_\alpha [a,b]  \\
&=&E[a,b]\\
&=&[a,b]\\
&=&[a,b]E\\
&=&w-\lim_\alpha [a,b]e_\alpha \\
&=&\lim_\alpha [a,b]e_\alpha ,
 \end{array}$$ 
 and by assumption  $A$ has an approximate identity $(e_\alpha)$.
\end{proof}
Note that the converse of this proposition is not true in general. For example in the sequel we see that every commutative Banach algebra is not biamenable. Although it may be unital or approximately unital.

For each integer $n\geq 0$ put $$AA^{(2n)}+A^{(2n)}A=\{aa^{(2n)}+b^{(2n)}b: a, b \in A, a^{(2n)}, b^{(2n)}\in A^{(2n)}\}.$$ Then we have the following proposition.
\begin{proposition}\label{n}
If $A$ is  biamenable, then for each integer $n\geq 0$, $span (AA^{(2n)}+A^{(2n)}A)$ is dense in $A^{(2n)}$.
\end{proposition}
\begin{proof}
If $span (AA^{(2n)}+A^{(2n)}A)$ is not dense in $A^{(2n)}$, then there exists a non-zero linear functional $f\in A^{(2n+1)}$ such that it is zero on $span (AA^{2n}+A^{(2n)}A)$. Now the bilinear mapping
$$\begin{array}{rcl}
D:A\times A&\rightarrow &A^{(2n+1)}\\
(a,b)&\mapsto &f(a)f(b)f
\end{array}$$
is a biderivation which is not inner. So $A$ is not biamenable, which is a contradiction. 
\end{proof}
Let $A$ be a Banach algebra and $$A^{n}=span\{a_1...a_n: a_i\in A\}\hspace{1cm}(n\in \mathbb{N}).$$
 As a corollary of the latter proposition we have:
\begin{corollary}\label{cor3}
If $A$ is  biamenable then for each $n\in \mathbb{N},$ $A^{n}$ is dense in $A$. 
\end{corollary}
\begin{proof}
By Proposition \ref{n} $A^{2}$ is dense in $A$. Now by applying the density of $A^{2}$ in $A$ we can prove that $A^{3}$ is dense in $A$ and also by an inductive method we can prove that for each $n$, $A^{n}$ is dense in $A$.
\end{proof}  
For a Banach algebra $A$, put $$[A,A]=\{[a,b]: a,b\in A\} \  \ \ \mbox{and} \ \ \  [A,A]A=\{[a,b]c: a, b, c\in A\}.$$ The following proposition gives a big class of non-biamenable Banach algebras.
\begin{proposition}
Let $A$ be a biamenable Banach algebra. Then $span([A,A]\cup [A,A]A)$ is dense in $A$.
\end{proposition}
\begin{proof}
Suppose $S=span([A,A]\cup [A,A]A)$ is not dense in $A$. Then there exists a nonzero element $f\in A^*$ such that $f|_S=0$. In particular for each $a, b, c\in A$, we have $f(ab)=f(ba)$ and $c.f(ab)=c.f(ba)$. Consider $X=\overline{f.A}$ as an $A$-bimodule with module actions
$$(f.a).b=f.ab, \ \ \  \mbox{and} \ \ \ b.(f.a)=0\hspace{1cm}(a, b\in A).$$ Then $Z(A,X^*)=\{0\}$ and so the only inner biderivation from $A\times A$ to $X^*$ is zero. Now by Corollary \ref{cor3} the bilinear mapping $D:A\times A\rightarrow X^*$ defined by $$\langle D(a,b),f.c\rangle=f(abc), \hspace{1cm}(a, b, c\in A)$$
is nonzero. Also for each $a, b, c, d\in A$ we have
$$\begin{array}{rcl}
\langle D(ab,c),f.d\rangle &=&f(abcd)\\
&=&f(bcda)\\
&=&\langle D(b,c),f.da\rangle \\
&=&\langle D(b,c),(f.d).a)\rangle \\
&=&\langle aD(b,c),f.d\rangle \\
&=&\langle aD(b,c)+D(a,c)b,f.d\rangle ,
\end{array}$$ and similarly
$$\begin{array}{rcl}
\langle D(a,bc),f.d\rangle &=&f(abcd)\\
&=&f(cdab)\\
&=&(b.f)(cda)\\
&=&(b.f)(acd)\\
&=&f(acdb)\\
&=&\langle D(a,c),f.db\rangle \\
&=&\langle bD(a,c)+D(a,b)c, f.d\rangle .
\end{array}$$
 So $D$ is a nonzero biderivation and so it is not inner. That is a contradiction.
\end{proof}
Note that if a biamenable Banach algebra $A$ has a right approximate identity, then $[A,A]\subseteq \overline{[A,A]A}$ and therefore $span([A,A]A)$ is dense in $A$. This may be compared with the converse of Proposition \ref{pr1}, for each biamenable Banach algebra.
\begin{corollary}\label{com}
Every non zero commutative Banach algebra is not biamenable.
\end{corollary}
This corollary implies that for every locally compact abelian group $G$, $L^1(G)$ is not biamenable but it is amenable by applying Johnson's theorem and the Example 1.1.5 of \cite{Ru}.
\begin{theorem} 
Let $A$ be a Banach algebra, $X$ be a Banach $A$-bimodule and $I$ be a closed ideal of it such that $Z(A,X^*)=Z(I,X^*)$. Then if $BH^1(I,X^*)=\{0\}$ and $\frac{A}{I}$ is biamenable, then $BH^1(A,X^*)=\{0\}$.
\end{theorem}
\begin{proof}
Let $D:A\times A\rightarrow X^*$ be a biderivation. Then $D_0=D|_{I\times I}\in BZ^1(I,X^*)$ and so $D_0=\Delta_E$, for some $E\in Z(A,X^*)$. Put $\tilde{D}=D-\Delta_E$. Then $\tilde{D}(I\times I)=0$ and so $\mathfrak{D}:\frac{A}{I}\times\frac{A}{I}\rightarrow X^*$, defined by $\mathfrak{D}((a+I,b+I))=\tilde{D}((a,b))$ is a well defined map. \\
Put $X_0=IX+XI$. Then $$(\frac{X}{X_0})^*=X_0^\perp=\{\phi\in X^*; \phi i=0=i\phi, \text{ for all } \ i\in I\}.$$
and so we can consider $X_0^\perp$ as an $\frac{A}{I}$-bimodule with the module actions $(a+I).\phi=a.\phi$ and $\phi.(a+I)=\phi.a$, for each $a\in A$ and $\phi\in X_0^\perp$. On the other hand we have
\begin{align*}
\tilde{D}(a,b)ij&=(\tilde{D}(ai,b)-a\tilde{D}(i,b))j\\
&=\tilde{D}(ai,bj)-b\tilde{D}(ai,j)-a\tilde{D}(i,bj)+ab\tilde{D}(i,j)\\
&=0, \quad (i, j\in I, a, b\in A).
\end{align*}
Similarly we can show that $ij\tilde{D}(a,b)=0$. Therefore by density of $I^2$ in $I$ (Proposition \ref{n}) we conclude that $\mathfrak{D}(\frac{A}{I}\times\frac{A}{I})\subseteq X_0^\perp$ and then we can coclude that $\mathfrak{D}:\frac{A}{I}\times\frac{A}{I}\rightarrow X_0^\perp$ is a biderivation. So there is $\psi\in Z(\frac{A}{I},X_0^\perp)\subseteq Z(A,X^*)$ such that $D-\Delta_E=\tilde{D}=\Delta_\psi$. Hence $D=\Delta_{E+\psi}$ and $E+\psi\in Z(A,X^*)$.
\end{proof}
We know that a Banach algebra $A$ is amenable if and only if the unconditional unitization $A^\flat$ of $A$ (see {\cite[Definition 1.3.3]{D}}) is amenable   {\cite[Corollary 2.3.11]{Ru}}. But it is not true for biamenability of Banach algebras. Indeed we show that the unconditional unitization of each Banach algebra is not biamenable.
\begin{lemma}
If $\theta: A\rightarrow B$ is a continuous homomorphism of Banach algebras with dense range and $A$ is biamenable, then so is $B$.
\end{lemma}
\begin{proof}
Let $X$ be a Banach $B$-bimodule. Consider $X$ as an $A$-bimodule with module actions $ax=\theta(a) x$ and $xa=x\theta(a)$. Now for each $D\in BZ^1(B,X^*)$, $D\circ (\theta\times\theta)\in BZ^1(A,X^*)$ and biamenability of $A$ implies that  $D\circ (\theta\times\theta)=\Delta_\phi$ for some $\phi\in Z(A,X^*)$. Now by density we conclude that $\phi\in Z(B,X^*)$ and  $D=\Delta_\phi$ .
\end{proof}
\begin{corollary}\label{co}
If $A$ is biamenable then for each closed ideal $I$ of $A$, $\frac{A}{I}$ is biamenable.
\end{corollary}
For analogues of the above two results in the area
 of amenability see Proposition 2.3.1 and Corollary 2.3.2 of \cite{Ru}.

Now we have the following theorem for the unconditional unitization $A^\flat$ of a Banach algebra $A$.
\begin{theorem}
For each Banach algebra $A$, the unconditional unitization $A^\flat$ is not biamenable.
\end{theorem}
\begin{proof}
If $A^\flat$ is biamenable then $\mathbb{C}=\frac{A^\flat}{A}$ is biamenable by
Corollary \ref{co} (recall that A is a closed ideal in $A^\flat$). A contradiction.
\end{proof}
\section{biamenability of a pair of Banach algebras}
Let $A$ and $B$ be Banach algebras and $X$ be an $A-B$-bimodule that is $X$ is an $A$-bimodule and $B$-bimodule and we have $$a(xb)=(ax)b \ \ , \ \ b(xa)=(bx)a\hspace{2cm}(a\in A, b\in B ,  x\in X).$$ The bounded bilinear mapping $D:A\times B \rightarrow X$ is called a biderivation if $D$ is a derivation  with respect to both arguments. That is the mappings $_aD:B\rightarrow X$ and $D_b:A\rightarrow X$ where $$_aD(b)=D(a,b)=D_b(a) \hspace{2cm}(a\in A , b\in B)$$ are derivations. We denote the space of such biderivations by $BZ^1(A\times B, X)$.

Let $x\in Z(A,X)\cap Z(B,X)$, where $$Z(A,X) =\{x\in X ; ax=xa \ \  \forall a\in A\}.$$
The map $D_x:A \times B\rightarrow X$ that
  $$D_x(a,b)=x[a,b]=(xa)b-(xb)a \ \ \ \ \ \ (a\in A, b\in B )$$
 is a basic example of biderivations and called an inner biderivation. We denote the space of such inner biderivations by $BN^1(A\times B,X)$. Also we define the first bicohomology group $BH^1(A\times B,X)$ as follows, $$BH^1(A\times B,X)=\frac{BZ^1(A\times B,X)}{BN^1(A\times B,X)}.$$
 
 Let $A$ and $B$ be Banach algebras. We say that $A$ and $B$ commute with respect to an $A-B$-bimodule $X$ if for each $a\in A, b\in B$ and $x\in X$ we have
$a(bx)=b(ax)$ and $(xb)a=(xa)b$.\\
Note that if $A$ and $B$ commute with respect to $X$ then they commute with respect to $A-B$-bimodule $X^*$.

For example if we consider $X$ as an $A-B$-bimodule with module actions zero on $A$, then  $A$ and $B$ commute with respect to this $A-B$-bimodule. Also if $A$ is commutative then $A$ commutes with itself with respect to $A$.
\begin{proposition}
If $A$ and $B$ commute with respect to an $A-B$-bimodule $X$ and there are nonzero derivations $d:A\rightarrow X$ and $d':B\rightarrow X$, then there is a nonzero biderivation from $A\times B$ into $X$.
\end{proposition}
\begin{proof}
Define $D:A\times B\rightarrow X$ by $D(a,b)=\delta_{d(a)}(b)+\delta_{d'(b)}(a)$.
\end{proof}
\begin{proposition} 
For each biderivation $D:A\times B\rightarrow X$, $$D(a,b)[c,d]=[a,b]D(c,d)\hspace{1cm}(a,c\in A , b,d\in B).$$
\end{proposition}
\begin{proof}
Since $X$ is an $A-B$-bimodule, we have for each $a,c\in A, b,d\in B$,
$$\begin{array}{rcl}
a(bD(c,d))+a(D(c,b)d)&+&(D(a,b)d)c+(bD(a,d))c\\&=&aD(c,bd)+D(a,bd)c\\
&=&D(ac,bd)\\&=&bD(ac,d)+D(ac,b)d\\
&=&b(aD(c,d))+b(D(a,d)c)+(aD(c,b))d+(D(a,b)c)d.
\end{array}$$
So $$[a,b]D(c,d)-D(a,b)[c,d]=a(bD(c,d))+b(aD(c,d)+(D(a,b)d)c-(D(a,b)c)d=0.$$
\end{proof}

 \begin{definition}
We say that the pair $(A,B)$ is  biamenable if for each $A-B$-bimodule $X$, $BH^1(A\times B, X^*)=\{0\}$.
\end{definition}
Obviously  biamenability of a Banach algebra $A$ is nothing else than some refinement of biamenability of the pair $(A,A)$ with the same module actions on each argument, Or briefly  biamenability of the pair $(A,A)$ implies  biamenability of the Banach algebra $A$.

\begin{remark}
If $Z(A,X)=\{0\}$ or $Z(B,X)=\{0\}$ then the only inner biderivation $D:A\times B\rightarrow X$ is zero. Therefore if $X$ is an $A-B$-bimodule such that there is a non zero biderivation from $A\times B$ into $X^*$ and $Z(A,X)=\{0\}$, then $(A, B)$ is not biamenable.
\end{remark}

The following theorem says that amenability of Banach algebras $A$ and $B$ are necessary conditions for biamenability of the pair $(A,B)$.

\begin{theorem}
\begin{itemize}
\item[(i)] If $\sigma(B)\neq\emptyset$ and $(A,B)$ is biamenable then $A$ is amenable. 
\item[(ii)] If $\sigma(A)\neq\emptyset$ and $(A,B)$ is biamenable then $B$ is amenable. 
\end{itemize}
\end{theorem}
\begin{proof}
We only prove (i).\\
Let $X$ be a Banach $A$-bimodule and $d:A\rightarrow X^*$ be a derivation. consider $B$-bimodule actions  $bx=0$ and  $xb=\varphi(b)x$, for some $\varphi\in \sigma(B)$.
Then it is easy to see that $X$ is a Banach $A-B$-bimodule.

Now $D:A \times B\rightarrow X^*$ with $D(a,b)=\varphi(b)d(a)$ is a biderivation and so $D(a,b)=x^*[a,b]=\delta_{x^*b}(a)$, for some $x^*\in Z(A,X^*)\cap Z(B,X^*)$. so $d(a)=\frac{D(a,b)}{\varphi(b)}=\delta_{\frac{x^*b}{\varphi(b)}}(a)$. for some $b\in B$ such that $\varphi(b)\neq 0.$ i.e. $A$ is amenable.
\end{proof}
Note that  the converse of this theorem is not true, in general. For example $\mathbb{C}$ is not biamenable by Corollary \ref{com} and so $(\mathbb{C},\mathbb{C})$ is not biamenable. Also $\sigma(\mathbb{C})\neq \emptyset$, but it is amenable.
\section{a condition for preserving  biamenability}
In this section we investigate the relation between  biamenability of two pairs $(A,B)$ and $(E,F)$ of Banach algebras, where there exist homomorphisms $T:E\rightarrow A$ and $S:F\rightarrow B$.

Let $A, B, E$ and $F$ be Banach algebras and $T:E\rightarrow A$ and $S:F\rightarrow B$ be bounded homomorphism of Banach algebras. If $X$ is a Banach $A-B$-bimodule, then we may consider $X$ as a Banach $E-F$-bimodule with module actions $e.x=T(e)x$, $x.e=xT(e)$, $f.x=S(f)x$ and $x.f=xS(f)$, which $x\in X$, $e\in E$ and $f\in F$.
\begin{theorem}
If $D:A\times B\rightarrow X$ is a biderivation, then $D\circ(T\times S):E\times F\rightarrow X$ is a biderivation. Also if $D$ is inner so is $D\circ(T\times S)$.

The converse is true if $T$ and $S$ are onto.
\end{theorem}
\begin{proof}
It is easy to check that for biderivation $D$, $D\circ(T\times S)$ is a biderivation. Also if $D$ is inner, then there is $x\in Z(A,X)\cap Z(B,X)$ such that $$D(a,b)=(xa)b-(xb)a\hspace{1cm}(a\in A, b\in B).$$ Therefore
$$D\circ(T\times S)(e,f)=(xT(e))S(f)-(xS(f))T(e)=(x.e).f-(x.f).e\hspace{1cm}(e\in E, f\in F).$$
Similarly we can show that $x\in Z(E,X)\cap Z(F,X)$.

For the next part note that if $T$ and $S$ are onto, then for each $a\in A$ and $b\in B$, $D(a,b)=D\circ(T\times S)(e,f)$ for some $e\in E$ and $f\in F$.
So for each $a, a'\in A$ and $b\in B$ there are $e, e'\in E$ and $f\in F$ such that
$$\begin{array}{rcl}
D(aa',b)&=&D\circ(T\times S)(ee',f)\\
&=&D\circ(T\times S)(e,f).e'+e.D\circ(T\times S)(e',f)\\
&=&D(a,b)T(e')+T(e)D(a',b)\\
&=&D(a,b)a'+aD(a',b).
\end{array}$$
Similarly we have for each $a\in A$ and $b, b'\in B$, $D(a,bb')=D(a,b)b'+bD(a,b').$
By a similar argument we may prove that if  $T$ and $S$ are onto and $D\circ(T\times S)$  is inner so is $D$.
\end{proof}
\begin{corollary}
Consider $A, B, E$ and $F$ as above. If  $T$ and $S$ are onto and $(E,F)$ is  biamenable, then so is $(A,B)$.
\end{corollary}
\bibliography{mybibfile}

\end{document}